\documentclass[reqno,12pt]{amsart}
\usepackage[T1]{fontenc}
\usepackage{enumerate}
\theoremstyle{plain}
\newtheorem{thm}{Theorem}

\begin{document}
\title{On rational systems in the plane. I. Riccati Cases.}

\author[Gabriel Lugo]{Gabriel Lugo}
\address{Department of Mathematics, University of Rhode Island,Kingston, RI 02881-0816, USA;}
\email{glugo@math.uri.edu}
\author[Frank J. Palladino]{Frank J. Palladino}
\address{Department of Mathematics, University of Rhode Island,Kingston, RI 02881-0816, USA;}
\email{frank@math.uri.edu}
\date{December 19, 2011}
\subjclass{39A10,39A11}
\keywords{difference equation, Riccati equation, global asymptotic stability, rational system}

\begin{abstract}
\noindent This paper is the first in a series of papers which will address, on a case by case basis, the special cases of the following rational system in the plane, labeled system \#11. 
$$x_{n+1}=\frac{\alpha_{1}}{A_{1}+y_{n}},\quad y_{n+1}=\frac{\alpha_{2}+\beta_{2}x_{n}+\gamma_{2}y_{n}}{A_{2}+B_{2}x_{n}+C_{2}y_{n}},\quad n=0,1,2,\dots ,$$
with $\alpha_{1},A_{1}>0$ and $\alpha_{2}, \beta_{2}, \gamma_{2}, A_{2}, B_{2}, C_{2}\geq 0$ and $\alpha_{2}+\beta_{2}+\gamma_{2}>0$ and $A_{2}+B_{2}+C_{2}>0$ and nonnegative initial conditions $x_{0}$ and $y_{0}$ so that the denominator is never zero.
In this article we focus on the special cases which are reducible to the Riccati difference equation. 
\end{abstract}
\maketitle

\section{Introduction}
This paper is the first in a series of papers which will address, on a case by case basis, the special cases of the following rational system in the plane, labeled system \#11. 
$$x_{n+1}=\frac{\alpha_{1}}{A_{1}+y_{n}},\quad y_{n+1}=\frac{\alpha_{2}+\beta_{2}x_{n}+\gamma_{2}y_{n}}{A_{2}+B_{2}x_{n}+C_{2}y_{n}},\quad n=0,1,2,\dots ,$$
with $\alpha_{1},A_{1}>0$ and $\alpha_{2}, \beta_{2}, \gamma_{2}, A_{2}, B_{2}, C_{2}\geq 0$ and $\alpha_{2}+\beta_{2}+\gamma_{2}>0$ and $A_{2}+B_{2}+C_{2}>0$ and nonnegative initial conditions $x_{0}$ and $y_{0}$ so that the denominator is never zero.
In this article we focus on the special cases which are reducible to the Riccati difference equation. The special cases of system \#11 which are Riccati or reducible to Riccati are the cases numbered $(11,1)$, $(11,2)$, $(11,3)$, $(11,4)$, $(11,5)$, $(11,7)$, $(11,9)$, $(11,10)$, $(11,11)$, $(11,13)$, $(11,17)$, $(11,19)$, $(11,20)$, $(11,22)$, $(11,24)$, $(11,28)$, and $(11,32)$, in the numbering system developed in \cite{cklm}. Our goal in this article is to determine a complete picture of the qualitative behavior for the cases in the above list to the best of our ability.\par
This article is organized as follows. In Section 2, we present a change of variables for the system numbered $(11,22)$. In the following sections we present a complete description of the qualitative behavior for the cases in the above list to the best of our ability. The cases are presented in ascending numerical order each case having its own theorem and appearing in its own section.

\section{A Change of Variables for the System (11,22)}
Consider the following system of rational difference equations, numbered (11,22) in the numbering system introduced in \cite{cklm}.
$$x_{n+1}=\frac{\alpha_{1}}{A_{1}+y_{n}},\quad n=0,1,\dots,$$
$$y_{n+1}=\alpha_{2}+\beta_{2}x_{n},\quad n=0,1,\dots.$$
With $\alpha_1,A_{1},\alpha_{2},\beta_{2}>0$ and arbitrary nonnegative initial conditions so that the denominator is never zero.
This system admits the following change of variables $y^{*}_{n}=y_{n}$ and $x^{*}_{n}=\beta_{2}x_{n}$, to obtain
$$x^{*}_{n+1}=\frac{\alpha_{1}\beta_{2}}{A_{1}+y^{*}_{n}},\quad n=0,1,\dots,$$
$$y^{*}_{n+1}=\alpha_{2}+x^{*}_{n},\quad n=0,1,\dots.$$
We may relabel our equation so that we have,
$$x_{n+1}=\frac{\alpha_{1}}{A_{1}+y_{n}},\quad n=0,1,\dots,$$
$$y_{n+1}=\alpha_{2}+x_{n},\quad n=0,1,\dots.$$
With $\alpha_1,A_{1},\alpha_{2}>0$ and arbitrary nonnegative initial conditions so that the denominator is never zero.

\section{The System (11,1)}
\begin{thm}
Consider the following system of rational difference equations
$$x_{n+1}=\frac{\alpha_{1}}{A_{1}+y_{n}},\quad n=0,1,\dots,$$
$$y_{n+1}=\frac{\alpha_{2}}{A_{2}},\quad n=0,1,\dots,$$
with $\alpha_1,A_{1},\alpha_{2},A_{2}>0$ and arbitrary nonnegative initial conditions so that the denominator is never zero. 
In this case we have a unique equilibrium 
$$(\bar{x},\bar{y})=\left(\frac{A_{2}\alpha_{1}}{A_{2}A_{1}+\alpha_{2}},\frac{\alpha_{2}}{A_{2}}\right).$$
Moreover, every solution is equal to the equilibrium after at most two steps.
\end{thm}
\begin{proof}
In this case we clearly have a unique equilibrium 
$$(\bar{x},\bar{y})=\left(\frac{A_{2}\alpha_{1}}{A_{2}A_{1}+\alpha_{2}},\frac{\alpha_{2}}{A_{2}}\right).$$
We have $(x_{1},y_{1})$ is on the line segment $\left(0,\frac{\alpha_{1}}{A_{1}}\right]\times\{\frac{\alpha_{2}}{A_{2}}\}$. Moreover 
$$(x_{n},y_{n})= (\bar{x},\bar{y})=\left(\frac{A_{2}\alpha_{1}}{A_{2}A_{1}+\alpha_{2}},\frac{\alpha_{2}}{A_{2}}\right),\quad n\geq 2.$$
\end{proof}

\section{The System (11,2)}
\begin{thm}
Consider the following system of rational difference equations
$$x_{n+1}=\frac{\alpha_{1}}{A_{1}+y_{n}},\quad n=0,1,\dots,$$
$$y_{n+1}=\frac{\alpha_{2}}{y_{n}},\quad n=0,1,\dots,$$
with $\alpha_1,A_{1},\alpha_{2}>0$ and arbitrary nonnegative initial conditions so that the denominator is never zero. In this case we have a unique equilibrium 
$$(\bar{x},\bar{y})=\left(\frac{\alpha_{1}}{A_{1}+\sqrt{\alpha_{2}}},\sqrt{\alpha_{2}}\right).$$
Moreover, $(x_{n},y_{n})=(x_{n+2},y_{n+2})$ for $n\geq 1$. Therefore every solution is eventually periodic with not necessarily prime period 2.
\end{thm}
\begin{proof}
To find the equilibria we solve:
$$\bar{x}=\frac{\alpha_{1}}{A_{1}+\bar{y}},$$
$$\bar{y}=\frac{\alpha_{2}}{\bar{y}}.$$
So $\bar{y}=\sqrt{\alpha_{2}}$, and $\bar{x}=\frac{\alpha_{1}}{A_{1}+\sqrt{\alpha_{2}}}$.
Thus, we have a unique equilibrium in this case.
Performing linearized stability analysis about the equilibrium $(\bar{x},\bar{y})$ we solve:
$$det\left(\begin{array}{cc}
-\lambda & \frac{-\bar{x}}{A_{1}+\bar{y}}\\               
0 & -1-\lambda\\ 
\end{array}
\right)= 0.$$
$$\lambda^{2}+\lambda =0.$$
So, $\lambda_{1}=0$ and $\lambda_{2}=-1$.
Notice that, $$y_{n+2}=\frac{\alpha_{2}}{y_{n+1}}=y_{n},\quad n=0,1,\dots.$$
Thus, $$x_{n+3}=\frac{\alpha_{1}}{A_{1}+y_{n+2}}=\frac{\alpha_{1}}{A_{1}+y_{n}}=x_{n+1},\quad n=0,1,\dots.$$
Thus, $(x_{n},y_{n})=(x_{n+2},y_{n+2})$ for $n\geq 1$. Therefore, every solution is eventually periodic with not necessarily prime period 2 for system (11,2).
\end{proof}

\section{The System (11,3)}
\begin{thm}
Consider the following system of rational difference equations
$$x_{n+1}=\frac{\alpha_{1}}{A_{1}+y_{n}},\quad n=0,1,\dots,$$
$$y_{n+1}=\frac{\alpha_{2}}{x_{n}},\quad n=0,1,\dots,$$
with $\alpha_1,A_{1},\alpha_{2}>0$ and arbitrary nonnegative initial conditions so that the denominator is never zero.
For this system of rational difference equations there are 2 regions in parametric space with distinct global behavior. The behavior is as follows:
\begin{enumerate}
\item If $\alpha_{1}>\alpha_{2}$, then the unique equilibrium $$(\bar{x},\bar{y})=\left(\frac{\alpha_{1}(\alpha_{1}-\alpha_{2})}{A_{1}(\alpha_{1}-\alpha_{2})+\alpha_{2}A_{1}},\frac{\alpha_{2}A_{1}}{\alpha_{1}-\alpha_{2}}\right),$$
is globally asymptotically stable.
\item If $\alpha_{1}\leq \alpha_{2}$, then there are no nonnegative equilibria and $\lim_{n\rightarrow \infty}(x_{n},y_{n})=(0,\infty)$ for all choices of initial conditions.
\end{enumerate}

\end{thm}
\begin{proof}
To find the equilibria we solve:
$$\bar{x}=\frac{\alpha_{1}}{A_{1}+\bar{y}},$$
$$\bar{y}=\frac{\alpha_{2}}{\bar{x}}.$$
So 
$$\bar{y}=\frac{\alpha_{2}}{\alpha_{1}}\left(A_{1}+\bar{y}\right),$$
thus, if $\alpha_{1}>\alpha_{2}$, then 
$$\bar{y}=\frac{\alpha_{2}A_{1}}{\alpha_{1}-\alpha_{2}}$$
and
$$\bar{x}=\frac{\alpha_{1}(\alpha_{1}-\alpha_{2})}{A_{1}(\alpha_{1}-\alpha_{2})+\alpha_{2}A_{1}}.$$
If $\alpha_{1}\leq \alpha_{2}$, then there are no nonnegative equilibria.
Performing linearized stability analysis about the equilibrium $(\bar{x},\bar{y})$ in the case $\alpha_{1}>\alpha_{2}$ we solve:
$$det\left(\begin{array}{cc}
-\lambda & \frac{-\bar{x}}{A_{1}+\bar{y}}\\               
\frac{-\bar{y}}{\bar{x}} & -\lambda\\ 
\end{array}
\right)= 0.$$
$$\lambda^{2}-\frac{\bar{y}}{A_{1}+\bar{y}}=0.$$
So, $$\lambda_{1,2}=\pm \sqrt{\frac{\alpha_{2}}{\alpha_{1}}}.$$
Thus, our equilibrium is locally asymptotically stable in this case.
Now we will show that our rational system on $(0,\frac{\alpha_{1}}{A_{1}})\times (0,\infty)$ is topologically conjugate to the following second order rational difference equation on $(0,\infty)^{2}$,
$$u_{n+1}=\frac{\alpha_{2}}{\left(\frac{\alpha_{1}}{A_{1}+u_{n-1}}\right)}=\frac{\alpha_{2}}{\alpha_{1}}\left(A_{1}+u_{n-1}\right),\quad n=0,1,\dots.$$
Let $f:(0,\frac{\alpha_{1}}{A_{1}})\times (0,\infty)\rightarrow (0,\frac{\alpha_{1}}{A_{1}})\times (0,\infty)$ be 
$$f((x,y))=\left(\frac{\alpha_{1}}{A_{1}+y},\frac{\alpha_{2}}{x}\right).$$
Let $g: (0,\infty)^{2}\rightarrow (0,\infty)^{2}$ be
$$g((x,y))=\left(\frac{\alpha_{2}}{\left(\frac{\alpha_{1}}{A_{1}+y}\right)},x\right).$$
Let $h:(0,\frac{\alpha_{1}}{A_{1}})\times (0,\infty)\rightarrow (0,\infty)^{2}$ be 
$$h((x,y))=\left(y,\frac{\alpha_{1}}{x}-A_{1}\right).$$
Let $h^{-1}:(0,\infty)^{2}\rightarrow (0,\frac{\alpha_{1}}{A_{1}})\times (0,\infty)$ be
 $$h^{-1}((x,y))=\left(\frac{\alpha_{1}}{A_{1}+y},x\right).$$
Now we will show that $h^{-1}\circ g \circ h = f$.
$$h^{-1}\circ g \circ h((x,y))=h^{-1}\circ g\left(\left(y,\frac{\alpha_{1}}{x}-A_{1}\right)\right)=h^{-1}\left(\left(\frac{\alpha_{2}}{x},y\right)\right)$$
$$= \left(\frac{\alpha_{1}}{A_{1}+y},\frac{\alpha_{2}}{x}\right)=f((x,y)).$$
So, since the system (11,3) is conjugate to the linear system above we see that the unique equilibrium is globally asymptotically stable in the case $\alpha_{1}>\alpha_{2}$.
Also, since the system (11,3) is conjugate to the linear system above we see that $\lim_{n\rightarrow \infty}(x_{n},y_{n})=(0,\infty)$ in the case $\alpha_{1}\leq\alpha_{2}$.
\end{proof}

\section{The System (11,4)}
\begin{thm}
Consider the following system of rational difference equations
$$x_{n+1}=\frac{\alpha_{1}}{A_{1}+y_{n}},\quad n=0,1,\dots,$$
$$y_{n+1}=\gamma_{2}y_{n},\quad n=0,1,\dots,$$
with $\alpha_1,A_{1},\gamma_{2}>0$ and arbitrary nonnegative initial conditions so that the denominator is never zero.
For this system of rational difference equations there are 3 regions in parametric space with distinct global behavior. The behavior is as follows:
\begin{enumerate}
\item If $\gamma_{2}>1$, then for any choice of initial conditions with $y_{0}\neq 0$, $$\lim_{n\rightarrow\infty}(x_{n},y_{n})=\left(0,\infty\right).$$
Moreover, $[0,\infty)\times \{0\}$ is the stable manifold for our saddle equilibrium $$(\bar{x},\bar{y})=\left(\frac{\alpha_{1}}{A_{1}},0\right).$$
\item If $\gamma_{2}=1$, then every point in the set
$$\left\{\left(\frac{\alpha_{1}}{A_{1}+v},v\right)|v\in [0,\infty )\right\},$$
is a nonhyperbolic equilibrium point. Furthermore, for each $v\in [0,\infty )$, the set $[0,\infty)\times \{v\}$ is a stable manifold for the nonhyperbolic equilibrium $\left(\frac{\alpha_{1}}{A_{1}+v},v\right)$.

\item If $\gamma_{2}<1$, the unique equilibrium $(\bar{x},\bar{y})=\left(\frac{\alpha_{1}}{A_{1}},0\right)$ is globally asymptotically stable.

\end{enumerate}
\end{thm}
\begin{proof}
To find the equilibria we solve:
$$\bar{x}=\frac{\alpha_{1}}{A_{1}+\bar{y}},$$
$$\bar{y}=\gamma_{2}\bar{y}.$$
So if $\gamma_{2}\neq 1$ then there is a unique equilibrium,
$$(\bar{x},\bar{y})=\left(\frac{\alpha_{1}}{A_{1}},0\right).$$
However if $\gamma_{2}=1$ then every point in the set
$$\left\{\left(\frac{\alpha_{1}}{A_{1}+v},v\right)|v\in [0,\infty )\right\},$$
is an equilibrium point.
\par\vspace{.4 cm} 

Performing linearized stability analysis about the equilibrium $(\bar{x},\bar{y})$ we solve:
$$det\left(\begin{array}{cc}
-\lambda & \frac{-\bar{x}}{A_{1}+\bar{y}}\\               
0 & \gamma_{2}-\lambda\\ 
\end{array}
\right)= 0.$$
$$\lambda^{2}-\gamma_{2}\lambda =0.$$
So $\lambda_{1}=0$ and $\lambda_{2}=\gamma_{2}$. So, the unique equilibrium is locally asymptotically stable when $\gamma_{2}<1$. The unique equilibrium is a saddle point when $\gamma_{2}>1$. Furthermore all equilibria are nonhyperbolic when $\gamma_{2}=1$.
\par\vspace{.4 cm}
Suppose $\gamma_{2}<1$, then clearly for an arbitrary choice of initial conditions $\lim_{n\rightarrow\infty}(x_{n},y_{n})=(\bar{x},\bar{y})=\left(\frac{\alpha_{1}}{A_{1}},0\right)$. So, when $\gamma_{2}<1$ the unique equilibrium is globally asymptotically stable.
\par\vspace{.4 cm}
Suppose $\gamma_{2}>1$, then clearly for any choice of initial conditions with $y_{0}\neq 0$, $\lim_{n\rightarrow\infty}(x_{n},y_{n})=(\bar{x},\bar{y})=\left(0,\infty\right)$.
If we choose $(x_{0},y_{0})\in [0,\infty)\times \{0\}$, then $(x_{n},y_{n})=(\bar{x},\bar{y})=\left(\frac{\alpha_{1}}{A_{1}},0\right)$ for $n\geq 1$. Thus in the case $\gamma_{2}>1$, $[0,\infty)\times \{0\}$ is the stable manifold for our saddle equilibrium.
\par\vspace{.4 cm}
Suppose $\gamma_{2}=1$, then clearly $(x_{n},y_{n})=\left(\frac{\alpha_{1}}{A_{1}+y_{0}},y_{0}\right)$ for $n\geq 1$. Thus every point in the set
$$\left\{\left(\frac{\alpha_{1}}{A_{1}+v},v\right)|v\in [0,\infty )\right\},$$
is a nonhyperbolic equilibrium point. Furthermore, for each $v\in [0,\infty )$, the set \newline$[0,\infty)\times \{v\}$ is a stable manifold for the nonhyperbolic equilibrium $\left(\frac{\alpha_{1}}{A_{1}+v},v\right)$.
\end{proof}

\section{The System (11,5)}
\begin{thm}
Consider the following system of rational difference equations
$$x_{n+1}=\frac{\alpha_{1}}{A_{1}+y_{n}},\quad n=0,1,\dots,$$
$$y_{n+1}=\frac{\beta_{2}}{B_{2}},\quad n=0,1,\dots,$$
with $\alpha_1,A_{1},\beta_{2},B_{2}>0$ and arbitrary nonnegative initial conditions so that the denominator is never zero.
In this case we have a unique equilibrium 
$$(\bar{x},\bar{y})=\left(\frac{B_{2}\alpha_{1}}{B_{2}A_{1}+\beta_{2}},\frac{\beta_{2}}{B_{2}}\right).$$
Moreover, every solution is equal to the equilibrium after at most two steps.
\end{thm}
\begin{proof}
In this case we clearly have a unique equilibrium 
$$(\bar{x},\bar{y})=\left(\frac{B_{2}\alpha_{1}}{B_{2}A_{1}+\beta_{2}},\frac{\beta_{2}}{B_{2}}\right).$$
We have $(x_{1},y_{1})$ is on the line segment $\left(0,\frac{\alpha_{1}}{A_{1}}\right]\times\{\frac{\beta_{2}}{B_{2}}\}$. Moreover 
$$(x_{n},y_{n})= (\bar{x},\bar{y})=\left(\frac{B_{2}\alpha_{1}}{B_{2}A_{1}+\beta_{2}},\frac{\beta_{2}}{B_{2}}\right),\quad n\geq 2.$$
\end{proof}

\section{The System (11,7)}
\begin{thm}
Consider the following system of rational difference equations
$$x_{n+1}=\frac{\alpha_{1}}{A_{1}+y_{n}},\quad n=0,1,\dots,$$
$$y_{n+1}=\beta_{2}x_{n},\quad n=0,1,\dots,$$
with $\alpha_1,A_{1},\beta_{2}>0$ and arbitrary nonnegative initial conditions so that the denominator is never zero.
In this case the unique nonnegative equilibrium $$(\bar{x},\bar{y})= \left(\frac{-A_{1}+\sqrt{A_{1}^{2}+4\beta_{2}\alpha_{1}}}{2\beta_{2}},\frac{-A_{1}+\sqrt{A_{1}^{2}+4\beta_{2}\alpha_{1}}}{2}\right)$$
is globally asymptotically stable.
\end{thm}
\begin{proof}
To find the equilibria we solve:
$$\bar{x}=\frac{\alpha_{1}}{A_{1}+\bar{y}},$$
$$\bar{y}=\beta_{2}\bar{x}.$$
So,
$$\bar{y}=\beta_{2}\frac{\alpha_{1}}{A_{1}+\bar{y}},$$
$$\bar{y}^{2}+A_{1}\bar{y}-\beta_{2}\alpha_{1}=0.$$
Hence 
$$\bar{y}_{1,2}= \frac{-A_{1}\pm\sqrt{A_{1}^{2}+4\beta_{2}\alpha_{1}}}{2}.$$
However, one of these roots is clearly negative so we have a unique nonnegative equilibrium,
$$(\bar{x},\bar{y})= \left(\frac{-A_{1}+\sqrt{A_{1}^{2}+4\beta_{2}\alpha_{1}}}{2\beta_{2}},\frac{-A_{1}+\sqrt{A_{1}^{2}+4\beta_{2}\alpha_{1}}}{2}\right).$$
Performing linearized stability analysis about the equilibrium $(\bar{x},\bar{y})$ we solve:
$$det\left(\begin{array}{cc}
-\lambda & \frac{-\bar{x}}{A_{1}+\bar{y}}\\               
\beta_{2} & -\lambda\\ 
\end{array}
\right)= 0.$$
$$\lambda^{2} + \frac{\beta_{2}\bar{x}}{A_{1}+\bar{y}}=\lambda^{2} + \frac{\bar{y}}{A_{1}+\bar{y}}=0.$$
So,
$$\lambda_{1,2}=\pm \sqrt{\frac{\bar{y}}{A_{1}+\bar{y}}}= \pm \sqrt{\frac{-A_{1}+\sqrt{A_{1}^{2}+4\beta_{2}\alpha_{1}}}{A_{1}+\sqrt{A_{1}^{2}+4\beta_{2}\alpha_{1}}}}.$$
Since our roots are inside the unit disk, our unique equilibrium is locally asymptotically stable.
Now we will show that our rational system on $(0,\frac{\alpha_{1}}{A_{1}})\times (0,\infty)$ is topologically conjugate to the following second order rational difference equation on $(0,\infty)^{2}$,
$$u_{n+1}= \beta_{2}\left(\frac{\alpha_{1}}{A_{1}+u_{n-1}}\right),\quad n=0,1,\dots.$$
Let $f:(0,\frac{\alpha_{1}}{A_{1}})\times (0,\infty)\rightarrow (0,\frac{\alpha_{1}}{A_{1}})\times (0,\infty)$ be 
$$f((x,y))=\left(\frac{\alpha_{1}}{A_{1}+y},\beta_{2}x\right).$$
Let $g: (0,\infty)^{2}\rightarrow (0,\infty)^{2}$ be
$$g((x,y))=\left( \beta_{2}\left(\frac{\alpha_{1}}{A_{1}+y}\right) ,x\right).$$
Let $h:(0,\frac{\alpha_{1}}{A_{1}})\times (0,\infty)\rightarrow (0,\infty)^{2}$ be 
$$h((x,y))=\left(y,\frac{\alpha_{1}}{x}-A_{1}\right).$$
Let $h^{-1}:(0,\infty)^{2}\rightarrow (0,\frac{\alpha_{1}}{A_{1}})\times (0,\infty)$ be
 $$h^{-1}((x,y))=\left(\frac{\alpha_{1}}{A_{1}+y},x\right).$$
Now we will show that $h^{-1}\circ g \circ h = f$.
$$h^{-1}\circ g \circ h((x,y))=h^{-1}\circ g\left(\left(y,\frac{\alpha_{1}}{x}-A_{1}\right)\right)=h^{-1}\left(\left(\beta_{2}x,y\right)\right)$$
$$= \left(\frac{\alpha_{1}}{A_{1}+y},\beta_{2}x\right)=f((x,y)).$$
Since our system is conjugate to a second order rational difference equation which decouples into two Riccati equations, we see that the unique equilibrium which we found earlier is globally asymptotically stable.
\end{proof}

\section{The System (11,9)}
\begin{thm}
Consider the following system of rational difference equations
$$x_{n+1}=\frac{\alpha_{1}}{A_{1}+y_{n}},\quad n=0,1,\dots,$$
$$y_{n+1}=\frac{\gamma_{2}}{C_{2}},\quad n=0,1,\dots,$$
with $\alpha_1,A_{1},\gamma_{2},C_{2}>0$ and arbitrary nonnegative initial conditions so that the denominator is never zero.
In this case we clearly have a unique equilibrium $$(\bar{x},\bar{y})=\left(\frac{C_{2}\alpha_{1}}{C_{2}A_{1}+\gamma_{2}},\frac{\gamma_{2}}{C_{2}}\right).$$
Moreover, every solution is equal to the equilibrium after at most two steps.
\end{thm}
\begin{proof}
In this case we clearly have a unique equilibrium 
$$(\bar{x},\bar{y})=\left(\frac{C_{2}\alpha_{1}}{C_{2}A_{1}+\gamma_{2}},\frac{\gamma_{2}}{C_{2}}\right).$$
We have $(x_{1},y_{1})$ is on the line segment $\left(0,\frac{\alpha_{1}}{A_{1}}\right]\times\{\frac{\gamma_{2}}{C_{2}}\}$. Moreover 
$$(x_{n},y_{n})= (\bar{x},\bar{y})=\left(\frac{C_{2}\alpha_{1}}{C_{2}A_{1}+\gamma_{2}},\frac{\gamma_{2}}{C_{2}}\right),\quad n\geq 2.$$
\end{proof}
\section{The System (11,10)}
\begin{thm}
Consider the following system of rational difference equations
$$x_{n+1}=\frac{\alpha_{1}}{A_{1}+y_{n}},\quad n=0,1,\dots,$$
$$y_{n+1}=\frac{\alpha_{2}}{A_{2}+y_{n}},\quad n=0,1,\dots,$$
with $\alpha_1,A_{1},\alpha_{2},A_{2}>0$ and arbitrary nonnegative initial conditions so that the denominator is never zero.
In this case, there is a unique nonnegative equilibrium which is globally asymptotically stable.
\end{thm}

\begin{proof}
To find the equilibria we solve:
$$\bar{x}=\frac{\alpha_{1}}{A_{1}+\bar{y}},$$
$$\bar{y}=\frac{\alpha_{2}}{A_{2}+\bar{y}}.$$
So, $$A_{2}\bar{y}+\bar{y}^{2}-\alpha_{2}=0.$$
Thus, $$\bar{y}=\frac{-A_{2}+\sqrt{A_{2}^{2}+4\alpha_{2}}}{2}$$
and $$\bar{x}=\frac{2\alpha_{1}}{2A_{1}-A_{2}+\sqrt{A_{2}^{2}+4\alpha_{2}}}.$$

Descarte's rule of signs tells us that there is a unique nonnegative equilibrium.
Performing linearized stability analysis about the equilibrium $(\bar{x},\bar{y})$ we solve:
$$det\left(\begin{array}{cc}
-\lambda & \frac{-\bar{x}}{A_{1}+\bar{y}}\\               
0 & \frac{-\bar{y}}{A_{2}+\bar{y}}-\lambda\\ 
\end{array}
\right)= 0.$$
$$\lambda^{2}+\left(\frac{\bar{y}}{A_{2}+\bar{y}}\right)\lambda =0.$$
We have $\lambda_{1}=0$ and $\lambda_{2}=\frac{-\bar{y}}{A_{2}+\bar{y}}$. Thus our unique nonnegative equilibrium is locally asymptotically stable.
\par\vspace{.4 cm}
Now we will show that our rational system on $(0,\frac{\alpha_{1}}{A_{1}})\times (0,\infty)$ is topologically conjugate to the following second order rational difference equation on $(0,\infty)^{2}$,
$$u_{n+1}=\frac{\alpha_{2}}{A_{2} + u_{n}},\quad n=0,1,\dots.$$
Let $f:(0,\frac{\alpha_{1}}{A_{1}})\times (0,\infty)\rightarrow (0,\frac{\alpha_{1}}{A_{1}})\times (0,\infty)$ be 
$$f((x,y))=\left(\frac{\alpha_{1}}{A_{1}+y},\frac{\alpha_{2}}{A_{2}+y}\right).$$
Let $g: (0,\infty)^{2}\rightarrow (0,\infty)^{2}$ be
$$g((x,y))=\left(\frac{\alpha_{2}}{A_{2} + x},x\right).$$
Let $h:(0,\frac{\alpha_{1}}{A_{1}})\times (0,\infty)\rightarrow (0,\infty)^{2}$ be 
$$h((x,y))=\left(y,\frac{\alpha_{1}}{x}-A_{1}\right).$$
Let $h^{-1}:(0,\infty)^{2}\rightarrow (0,\frac{\alpha_{1}}{A_{1}})\times (0,\infty)$ be
 $$h^{-1}((x,y))=\left(\frac{\alpha_{1}}{A_{1}+y},x\right).$$
Now we will show that $h^{-1}\circ g \circ h = f$.
$$h^{-1}\circ g \circ h((x,y))=h^{-1}\circ g\left(\left(y,\frac{\alpha_{1}}{x}-A_{1}\right)\right)=h^{-1}\left(\left(\frac{\alpha_{2}}{A_{2}+y},y\right)\right)$$
$$= \left(\frac{\alpha_{1}}{A_{1}+y},\frac{\alpha_{2}}{A_{2}+y}\right)=f((x,y)).$$
Since our system (11,10) is conjugate to a Riccati difference equation we see that every solution converges to the unique nonnegative equilibrium. Thus, since we have already shown the unique nonnegative equilibrium to be locally asymptotically stable, the unique nonnegative equilibrium is globally asymptotically stable. 
\end{proof}

\section{The System (11,11)}
\begin{thm}
Consider the following system of rational difference equations
$$x_{n+1}=\frac{\alpha_{1}}{A_{1}+y_{n}},\quad n=0,1,\dots,$$
$$y_{n+1}=\frac{\alpha_{2}}{A_{2}+x_{n}},\quad n=0,1,\dots,$$
with $\alpha_1,A_{1},\alpha_{2},A_{2}>0$ and arbitrary nonnegative initial conditions so that the denominator is never zero.
In this case, the unique nonnegative equilibrium is globally asymptotically stable.
\end{thm}
\begin{proof}
To find the equilibria we solve:
$$\bar{x}=\frac{\alpha_{1}}{A_{1}+\bar{y}},$$
$$\bar{y}=\frac{\alpha_{2}}{A_{2}+\bar{x}}.$$
So,
$$A_{2}\bar{y}-\alpha_{2}=A_{1}\bar{x}-\alpha_{1},$$
$$A_{2}\bar{y}+\bar{x}\bar{y}=A_{1}\bar{x}+\alpha_{2}-\alpha_{1}+\frac{A_{1}\bar{x}^{2}+(\alpha_{2}-\alpha_{1})\bar{x}}{A_{2}},$$
$$0=A_{1}\bar{x}-\alpha_{1}+\frac{A_{1}\bar{x}^{2}+(\alpha_{2}-\alpha_{1})\bar{x}}{A_{2}},$$
$$\frac{A_{1}}{A_{2}}\bar{x}^{2}+\frac{A_{1}A_{2}+\alpha_{2}-\alpha_{1}}{A_{2}}\bar{x} -\alpha_{1}=0.$$
So, by Descarte's rule of signs we see that there is a unique nonnegative equilibrium $(\bar{x},\bar{y})$.
Performing linearized stability analysis about the equilibrium $(\bar{x},\bar{y})$ we solve:
$$det\left(\begin{array}{cc}
-\lambda & \frac{-\bar{x}}{A_{1}+\bar{y}}\\               
\frac{-\bar{y}}{A_{2}+\bar{x}} & -\lambda\\ 
\end{array}
\right)= 0.$$
$$\lambda^{2}- \frac{\bar{x}\bar{y}}{\left(A_{1}+\bar{y}\right)\left(A_{2}+\bar{x}\right)}=0.$$
So,
$$\lambda_{1,2}=\pm \sqrt{\frac{\bar{x}\bar{y}}{\left(A_{1}+\bar{y}\right)\left(A_{2}+\bar{x}\right)}}.$$
Thus our unique nonnegative equilibrium is locally asymptotically stable.
\par\vspace{.4 cm}
Now we will show that our rational system on $(0,\frac{\alpha_{1}}{A_{1}})\times (0,\infty)$ is topologically conjugate to the following second order rational difference equation on $(0,\infty)^{2}$,
$$u_{n+1}=\frac{\alpha_{2}}{A_{2}+\left(\frac{\alpha_{1}}{A_{1}+u_{n-1}}\right)},\quad n=0,1,\dots.$$
Let $f:(0,\frac{\alpha_{1}}{A_{1}})\times (0,\infty)\rightarrow (0,\frac{\alpha_{1}}{A_{1}})\times (0,\infty)$ be 
$$f((x,y))=\left(\frac{\alpha_{1}}{A_{1}+y},\frac{\alpha_{2}}{A_{2}+x}\right).$$
Let $g: (0,\infty)^{2}\rightarrow (0,\infty)^{2}$ be
$$g((x,y))=\left(\frac{\alpha_{2}}{A_{2}+ \left(\frac{\alpha_{1}}{A_{1}+y}\right) },x\right).$$
Let $h:(0,\frac{\alpha_{1}}{A_{1}})\times (0,\infty)\rightarrow (0,\infty)^{2}$ be 
$$h((x,y))=\left(y,\frac{\alpha_{1}}{x}-A_{1}\right).$$
Let $h^{-1}:(0,\infty)^{2}\rightarrow (0,\frac{\alpha_{1}}{A_{1}})\times (0,\infty)$ be
 $$h^{-1}((x,y))=\left(\frac{\alpha_{1}}{A_{1}+y},x\right).$$
Now we will show that $h^{-1}\circ g \circ h = f$.
$$h^{-1}\circ g \circ h((x,y))=h^{-1}\circ g\left(\left(y,\frac{\alpha_{1}}{x}-A_{1}\right)\right)=h^{-1}\left(\left(\frac{\alpha_{2}}{A_{2}+x},y\right)\right)$$
$$= \left(\frac{\alpha_{1}}{A_{1}+y},\frac{\alpha_{2}}{A_{2}+x}\right)=f((x,y)).$$
\par\vspace{.4 cm}
Since the system (11,11) is conjugate to a second order rational difference equation which decouples into two Riccati difference equations with Riccati number different from zero, the unique nonnegative equilibrium is globally asymptotically stable for the system (11,11).
\end{proof}

\section{The System (11,13)}
\begin{thm}
Consider the following system of rational difference equations
\par\vspace{.6 cm}
$$x_{n+1}=\frac{\alpha_{1}}{A_{1}+y_{n}},\quad n=0,1,\dots,$$
$$y_{n+1}=\frac{y_{n}}{A_{2}+y_{n}},\quad n=0,1,\dots,$$
\par\vspace{.4 cm}
with $\alpha_1,A_{1},A_{2}>0$ and arbitrary nonnegative initial conditions so that the denominator is never zero.
For this system of rational difference equations there are 2 regions in parametric space with distinct global behavior. The behavior is as follows:
\begin{enumerate}
\item If $A_{2}\geq 1$ the unique equilibrium $(\frac{\alpha_{1}}{A_{1}},0)$ is globally asymptotically stable. 
\item If $A_{2}< 1$, then there are exactly two distinct equilibria. The equilibrium $\left(\frac{\alpha_{1}}{A_{1}},0\right)$ is a saddle with $[0,\infty)\times\{0\}$ as its stable manifold. Moreover, the equilibrium $(\frac{\alpha_{1}}{A_{1}+1-A_{2}},1-A_{2})$ is locally asymptotically stable with attracting basin $[0,\infty)\times(0,\infty)$.
\end{enumerate}

\end{thm}
\begin{proof}
To find the equilibria we solve:
$$\bar{x}=\frac{\alpha_{1}}{A_{1}+\bar{y}},$$
$$\bar{y}=\frac{\bar{y}}{A_{2}+\bar{y}}.$$
So, $(\bar{x}_{1},\bar{y}_{1})=(\frac{\alpha_{1}}{A_{1}},0)$ and if $A_{2}<1$ and $\bar{y}\neq 0$, then 
$$(\bar{x}_{2},\bar{y}_{2})=\left(\frac{\alpha_{1}}{A_{1}+1-A_{2}},1-A_{2}\right).$$
\par\vspace{.4 cm}

Performing linearized stability analysis about the equilibrium $(\bar{x},\bar{y})$ we solve:
$$det\left(\begin{array}{cc}
-\lambda & \frac{-\bar{x}}{A_{1}+\bar{y}}\\               
0 & \frac{1-\bar{y}}{A_{2}+\bar{y}}-\lambda\\ 
\end{array}
\right)= 0.$$
$$\lambda^{2}-\left(\frac{1-\bar{y}}{A_{2}+\bar{y}}\right)\lambda =0.$$
So,
$$\lambda_{1}=0,\quad \lambda_{2}= \frac{1-\bar{y}}{A_{2}+\bar{y}}$$
So, if $A_{2}<1$, then the equilibrium $(\frac{\alpha_{1}}{A_{1}},0)$ is a saddle and the equilibrium $(\frac{\alpha_{1}}{A_{1}+1-A_{2}},1-A_{2})$ is locally asymptotically stable.
If $A_{2}=1$, then the unique equilibrium $(\frac{\alpha_{1}}{A_{1}},0)$ is nonhyperbolic.
If $A_{2}>1$, then the unique equilibrium $(\frac{\alpha_{1}}{A_{1}},0)$ is locally asymptotically stable.
Now we will show that our rational system on $(0,\frac{\alpha_{1}}{A_{1}})\times (0,\infty)$ is topologically conjugate to the following second order rational difference equation on $(0,\infty)^{2}$,
$$u_{n+1}=\frac{u_{n}}{A_{2}+ u_{n}},\quad n=0,1,\dots.$$
Let $f:(0,\frac{\alpha_{1}}{A_{1}})\times (0,\infty)\rightarrow (0,\frac{\alpha_{1}}{A_{1}})\times (0,\infty)$ be 
$$f((x,y))=\left(\frac{\alpha_{1}}{A_{1}+y},\frac{y}{A_{2}+y}\right).$$
Let $g: (0,\infty)^{2}\rightarrow (0,\infty)^{2}$ be
$$g((x,y))=\left(\frac{x}{A_{2}+ x},x\right).$$
Let $h:(0,\frac{\alpha_{1}}{A_{1}})\times (0,\infty)\rightarrow (0,\infty)^{2}$ be 
$$h((x,y))=\left(y,\frac{\alpha_{1}}{x}-A_{1}\right).$$
Let $h^{-1}:(0,\infty)^{2}\rightarrow (0,\frac{\alpha_{1}}{A_{1}})\times (0,\infty)$ be
 $$h^{-1}((x,y))=\left(\frac{\alpha_{1}}{A_{1}+y},x\right).$$
Now we will show that $h^{-1}\circ g \circ h = f$.
$$h^{-1}\circ g \circ h((x,y))=h^{-1}\circ g\left(\left(y,\frac{\alpha_{1}}{x}-A_{1}\right)\right)=h^{-1}\left(\left(\frac{y}{A_{2}+y},y\right)\right)$$
$$= \left(\frac{\alpha_{1}}{A_{1}+y},\frac{y}{A_{2}+y}\right)=f((x,y)).$$
\par\vspace{.4 cm}
From the conjugacy we see that for system (11,13) whenever $A_{2}\geq 1$, the unique equilibrium $(\frac{\alpha_{1}}{A_{1}},0)$ is globally asymptotically stable. 
Furthermore, whenever $A_{2}< 1$, the equilibrium $(\frac{\alpha_{1}}{A_{1}},0)$ is a saddle with $[0,\infty)\times\{0\}$ as its stable manifold. Also, whenever $A_{2}< 1$, the equilibrium $(\frac{\alpha_{1}}{A_{1}+1-A_{2}},1-A_{2})$ is locally asymptotically stable with attracting basin $[0,\infty)\times(0,\infty)$. 
\end{proof}

\section{The System (11,17)}
\begin{thm}
Consider the following system of rational difference equations
$$x_{n+1}=\frac{\alpha_{1}}{A_{1}+y_{n}},\quad n=0,1,\dots,$$
$$y_{n+1}=\frac{x_{n}}{A_{2}+x_{n}},\quad n=0,1,\dots,$$
with $\alpha_1,A_{1},A_{2}>0$ and arbitrary nonnegative initial conditions so that the denominator is never zero.
In this case, the unique nonnegative equilibrium
$$\bar{x}=\frac{\alpha_{1}-A_{1}A_{1}+\sqrt{\left(\alpha_{1}-A_{1}A_{1}\right)^{2}+4\alpha_{1}A_{2}+4\alpha_{1}A_{1}A_{2}}}{2+2A_{1}},$$
$$\bar{y}=\frac{2\alpha_{1}+2A_{1}\alpha_{1}}{\alpha_{1}-A_{1}A_{1}+\sqrt{\left(\alpha_{1}-A_{1}A_{1}\right)^{2}+4\alpha_{1}A_{2}+4\alpha_{1}A_{1}A_{2}}}-A_{1},$$
is globally asymptotically stable.
\end{thm}
\begin{proof}
To find the equilibria we solve:
$$\bar{x}=\frac{\alpha_{1}}{A_{1}+\bar{y}},$$
$$\bar{y}=\frac{\bar{x}}{A_{2}+\bar{x}}.$$
So,
$$\bar{x}=\frac{\alpha_{1}A_{2}+\alpha_{1}\bar{x}}{A_{1}A_{2}+A_{1}\bar{x}+\bar{x}}.$$
Thus,
$$A_{1}A_{2}\bar{x}+(A_{1}+1)\bar{x}^{2}=\alpha_{1}A_{2}+\alpha_{1}\bar{x}.$$
Hence,
$$\bar{x}=\frac{\alpha_{1}-A_{1}A_{1}+\sqrt{\left(\alpha_{1}-A_{1}A_{1}\right)^{2}+4\alpha_{1}A_{2}+4\alpha_{1}A_{1}A_{2}}}{2+2A_{1}}.$$
So,
$$\bar{y}=\frac{2\alpha_{1}+2A_{1}\alpha_{1}}{\alpha_{1}-A_{1}A_{1}+\sqrt{\left(\alpha_{1}-A_{1}A_{1}\right)^{2}+4\alpha_{1}A_{2}+4\alpha_{1}A_{1}A_{2}}}-A_{1}.$$
\par\vspace{.4 cm}
Performing linearized stability analysis about the equilibrium $(\bar{x},\bar{y})$ we solve:
$$det\left(\begin{array}{cc}
-\lambda & \frac{-\bar{x}}{A_{1}+\bar{y}}\\               
\frac{1-\bar{y}}{A_{2}+\bar{x}} & -\lambda\\ 
\end{array}
\right)= 0.$$
$$\lambda^{2}- \frac{\bar{x}\bar{y}-\bar{x}}{\left(A_{1}+\bar{y}\right)\left(A_{2}+\bar{x}\right)}=0.$$
Since $\bar{x}\bar{y}=\bar{x}-A_{2}\bar{y}$ from the equilibrium equations, we get
$$\lambda^{2}+ \frac{A_{2}\bar{y}}{\left(A_{1}+\bar{y}\right)\left(A_{2}+\bar{x}\right)}=0.$$
Thus both roots lie inside the unit disk and our unique equilibrium, $(\bar{x},\bar{y})$, is locally asymptotically stable.
\par\vspace{.4 cm}
Now we will show that our rational system on $(0,\frac{\alpha_{1}}{A_{1}})\times (0,\infty)$ is topologically conjugate to the following second order rational difference equation on $(0,\infty)^{2}$,
$$u_{n+1}=\frac{\left(\frac{\alpha_{1}}{A_{1}+u_{n-1}}\right)}{A_{2}+ \left(\frac{\alpha_{1}}{A_{1}+u_{n-1}}\right) },\quad n=0,1,\dots.$$
Let $f:(0,\frac{\alpha_{1}}{A_{1}})\times (0,\infty)\rightarrow (0,\frac{\alpha_{1}}{A_{1}})\times (0,\infty)$ be 
$$f((x,y))=\left(\frac{\alpha_{1}}{A_{1}+y},\frac{x}{A_{2}+x}\right).$$
Let $g: (0,\infty)^{2}\rightarrow (0,\infty)^{2}$ be
$$g((x,y))=\left(\frac{\left(\frac{\alpha_{1}}{A_{1}+y}\right) }{A_{2}+ \left(\frac{\alpha_{1}}{A_{1}+y}\right) },x\right).$$
Let $h:(0,\frac{\alpha_{1}}{A_{1}})\times (0,\infty)\rightarrow (0,\infty)^{2}$ be 
$$h((x,y))=\left(y,\frac{\alpha_{1}}{x}-A_{1}\right).$$
Let $h^{-1}:(0,\infty)^{2}\rightarrow (0,\frac{\alpha_{1}}{A_{1}})\times (0,\infty)$ be
 $$h^{-1}((x,y))=\left(\frac{\alpha_{1}}{A_{1}+y},x\right).$$
Now we will show that $h^{-1}\circ g \circ h = f$.
$$h^{-1}\circ g \circ h((x,y))=h^{-1}\circ g\left(\left(y,\frac{\alpha_{1}}{x}-A_{1}\right)\right)=h^{-1}\left(\left(\frac{x}{A_{2}+x},y\right)\right)$$
$$= \left(\frac{\alpha_{1}}{A_{1}+y},\frac{x}{A_{2}+x}\right)=f((x,y)).$$
\par\vspace{.4 cm}
Since the system (11,17) is conjugate to a second order rational difference equation which decouples into two Riccati difference equations with Riccati number different from zero, the unique nonnegative equilibrium is globally asymptotically stable for the system (11,17).
\end{proof}

\section{The System (11,19)}
\begin{thm}
Consider the following system of rational difference equations
$$x_{n+1}=\frac{\alpha_{1}}{A_{1}+y_{n}},\quad n=0,1,\dots,$$
$$y_{n+1}=\alpha_{2}+\gamma_{2}y_{n},\quad n=0,1,\dots,$$
with $\alpha_1,A_{1},\alpha_{2},\gamma_{2}>0$ and arbitrary nonnegative initial conditions so that the denominator is never zero.
For this system of rational difference equations there are 2 regions in parametric space with distinct global behavior. The behavior is as follows:
\begin{enumerate}
\item If $\gamma_{2}<1$, then the unique equilibrium $$\left(\frac{\alpha_{1}-\alpha_{1}\gamma_{2}}{A_{1}-A_{1}\gamma_{2}+\alpha_{2}},\frac{\alpha_{2}}{1-\gamma_{2}}\right),$$ is globally asymptotically stable.
\item If $\gamma_{2}\geq 1$, then for an arbitrary choice of initial conditions $$\lim_{n\rightarrow\infty}(x_{n},y_{n})=\left(0,\infty\right).$$
\end{enumerate}
\end{thm}
\begin{proof}
If $\gamma_{2}\geq 1$, then clearly $\lim_{n\rightarrow\infty}y_{n}=\infty$ and $\lim_{n\rightarrow\infty}x_{n}=0$.
If $\gamma_{2}< 1$, then we find the equilibrium by solving the following system of equations:
$$\bar{x}=\frac{\alpha_{1}}{A_{1}+\bar{y}},$$
$$\bar{y}=\alpha_{2}+\gamma_{2}\bar{y}.$$
So,
$$\bar{y}=\frac{\alpha_{2}}{1-\gamma_{2}},$$
$$\bar{x}=\frac{\alpha_{1}-\alpha_{1}\gamma_{2}}{A_{1}-A_{1}\gamma_{2}+\alpha_{2}}.$$
\par\vspace{.4 cm}

Performing linearized stability analysis about the equilibrium $(\bar{x},\bar{y})$ we solve:
$$det\left(\begin{array}{cc}
-\lambda & \frac{-\bar{x}}{A_{1}+\bar{y}}\\               
0 & \gamma_{2}-\lambda\\ 
\end{array}
\right)= 0.$$
$$\lambda^{2}-\gamma_{2}\lambda=0.$$
So, in the case $\gamma_{2} < 1$, the unique equilibrium,
$$\left(\frac{\alpha_{1}-\alpha_{1}\gamma_{2}}{A_{1}-A_{1}\gamma_{2}+\alpha_{2}},\frac{\alpha_{2}}{1-\gamma_{2}}\right),$$
is locally asymptotically stable.
\par\vspace{.4 cm}
Now, we will show that our rational system on $(0,\frac{\alpha_{1}}{A_{1}})\times (0,\infty)$ is topologically conjugate to the following second order rational difference equation on $(0,\infty)^{2}$,
$$u_{n+1}=\alpha_{2}+ \gamma_{2}u_{n},\quad n=0,1,\dots.$$
Let $f:(0,\frac{\alpha_{1}}{A_{1}})\times (0,\infty)\rightarrow (0,\frac{\alpha_{1}}{A_{1}})\times (0,\infty)$ be 
$$f((x,y))=\left(\frac{\alpha_{1}}{A_{1}+y},\alpha_{2}+\gamma_{2}y\right).$$
Let $g: (0,\infty)^{2}\rightarrow (0,\infty)^{2}$ be
$$g((x,y))=\left(\alpha_{2}+ \gamma_{2}x,x\right).$$
Let $h:(0,\frac{\alpha_{1}}{A_{1}})\times (0,\infty)\rightarrow (0,\infty)^{2}$ be 
$$h((x,y))=\left(y,\frac{\alpha_{1}}{x}-A_{1}\right).$$
Let $h^{-1}:(0,\infty)^{2}\rightarrow (0,\frac{\alpha_{1}}{A_{1}})\times (0,\infty)$ be
 $$h^{-1}((x,y))=\left(\frac{\alpha_{1}}{A_{1}+y},x\right).$$
Now we will show that $h^{-1}\circ g \circ h = f$.
$$h^{-1}\circ g \circ h((x,y))=h^{-1}\circ g\left(\left(y,\frac{\alpha_{1}}{x}-A_{1}\right)\right)=h^{-1}\left(\left(\alpha_{2}+\gamma_{2}y,y\right)\right)$$
$$= \left(\frac{\alpha_{1}}{A_{1}+y},\alpha_{2}+\gamma_{2}y\right)=f((x,y)).$$
Since our system (11,19) is conjugate to a first order linear difference equation, we get that when $\gamma_{2}<1$ the unique equilibrium is globally asymptotically stable. Furthermore, when
$\gamma_{2}\geq 1$ for our system (11,19), then $\lim_{n\rightarrow\infty}y_{n}=\infty$ and $\lim_{n\rightarrow\infty}x_{n}=0$.
\end{proof}

\section{The System (11,20)}
\begin{thm}
Consider the following system of rational difference equations
$$x_{n+1}=\frac{\alpha_{1}}{A_{1}+y_{n}},\quad n=0,1,\dots,$$
$$y_{n+1}=\frac{\alpha_{2}+\gamma_{2}y_{n}}{y_{n}},\quad n=0,1,\dots,$$
with $\alpha_1,A_{1},\alpha_{2},\gamma_{2}>0$ and arbitrary nonnegative initial conditions so that the denominator is never zero.
In this case, the unique equilibrium
$$\left(\frac{2\alpha_{1}}{2A_{1}+\gamma_{2}+\sqrt{\gamma_{2}^{2}+4\alpha_{2}}},\frac{\gamma_{2}+\sqrt{\gamma_{2}^{2}+4\alpha_{2}}}{2}\right),$$
is globally asymptotically stable.
\end{thm}
\begin{proof}
To find the equilibria we solve:
$$\bar{x}=\frac{\alpha_{1}}{A_{1}+\bar{y}},$$
$$\bar{y}=\frac{\alpha_{2}+\gamma_{2}\bar{y}}{\bar{y}}.$$
So,
$$\bar{y}=\frac{\gamma_{2}+\sqrt{\gamma_{2}^{2}+4\alpha_{2}}}{2},$$
$$\bar{x}=\frac{2\alpha_{1}}{2A_{1}+\gamma_{2}+\sqrt{\gamma_{2}^{2}+4\alpha_{2}}}.$$
\par\vspace{.4 cm}
Performing linearized stability analysis about the equilibrium $(\bar{x},\bar{y})$ we solve:
$$det\left(\begin{array}{cc}
-\lambda & \frac{-\bar{x}}{A_{1}+\bar{y}}\\               
0 & \frac{\gamma_{2}-\bar{y}}{\bar{y}}-\lambda\\ 
\end{array}
\right)= 0.$$
$$\lambda^{2}-\left(\frac{\gamma_{2}-\bar{y}}{\bar{y}}\right)\lambda =0.$$
Since $0<\gamma_{2}< \bar{y}$, we have that $\left|\frac{\gamma_{2}}{\bar{y}}-1\right|<1$. Thus, the unique equilibrium,
$$\left(\frac{2\alpha_{1}}{2A_{1}+\gamma_{2}+\sqrt{\gamma_{2}^{2}+4\alpha_{2}}},\frac{\gamma_{2}+\sqrt{\gamma_{2}^{2}+4\alpha_{2}}}{2}\right),$$
is locally asymptotically stable.
\par\vspace{.4 cm}
Since the recursive equation governing $y_{n}$ is a Riccati equation which does not depend on $x_{n}$ in any way, the unique equilibrium is globally asymptotically stable for the system (11,20).
\end{proof}

\section{The System (11,22)}
\begin{thm}
Consider the following system of rational difference equations
$$x_{n+1}=\frac{\alpha_{1}}{A_{1}+y_{n}},\quad n=0,1,\dots,$$
$$y_{n+1}=\alpha_{2}+x_{n},\quad n=0,1,\dots,$$
with $\alpha_1,A_{1},\alpha_{2}>0$ and arbitrary nonnegative initial conditions so that the denominator is never zero.
In this case, the unique nonnegative equilibrium
$$\left(\frac{-\alpha_{2}-A_{1}+\sqrt{\left(\alpha_{2}+A_{1}\right)^{2}+4\alpha_{1}}}{2},\frac{\alpha_{2}-A_{1}+\sqrt{\left(\alpha_{2}+A_{1}\right)^{2}+4\alpha_{1}}}{2}\right)$$
is globally asymptotically stable.
\end{thm}
\begin{proof}
To find the equilibria we solve:
$$\bar{x}=\frac{\alpha_{1}}{A_{1}+\bar{y}},$$
$$\bar{y}=\alpha_{2}+\bar{x}.$$
So,
$$\bar{y}=\alpha_{2}+\frac{\alpha_{1}}{A_{1}+\bar{y}},$$
$$\bar{y}^{2}+A_{1}\bar{y}=\alpha_{2}A_{1}+\alpha_{2}\bar{y}+\alpha_{1}.$$
Thus,
$$\bar{y}=\frac{\alpha_{2}-A_{1}+\sqrt{\left(\alpha_{2}+A_{1}\right)^{2}+4\alpha_{1}}}{2},$$
$$\bar{x}=\frac{-\alpha_{2}-A_{1}+\sqrt{\left(\alpha_{2}+A_{1}\right)^{2}+4\alpha_{1}}}{2},$$
\par\vspace{.4 cm}
Performing linearized stability analysis about the equilibrium $(\bar{x},\bar{y})$ we solve:
$$det\left(\begin{array}{cc}
-\lambda & \frac{-\bar{x}}{A_{1}+\bar{y}}\\               
1 & -\lambda\\ 
\end{array}
\right)= 0.$$
$$\lambda^{2} + \frac{\bar{x}}{A_{1}+\bar{y}}=0.$$
Since $\bar{y}=\alpha_{2}+\bar{x}$, we have
$$\lambda^{2} + \frac{\bar{x}}{A_{1}+\alpha_{2}+\bar{x}}=0.$$
Thus, our unique nonnegative equilibrium is locally asymptotically stable.
\par\vspace{.4 cm}
Now we will show that our rational system on $(0,\frac{\alpha_{1}}{A_{1}})\times (0,\infty)$ is topologically conjugate to the following second order rational difference equation on $(0,\infty)^{2}$,
$$u_{n+1}=\alpha_{2}+ \left(\frac{\alpha_{1}}{A_{1}+u_{n-1}}\right) ,\quad n=0,1,\dots.$$
Let $f:(0,\frac{\alpha_{1}}{A_{1}})\times (0,\infty)\rightarrow (0,\frac{\alpha_{1}}{A_{1}})\times (0,\infty)$ be 
$$f((x,y))=\left(\frac{\alpha_{1}}{A_{1}+y},\alpha_{2}+x\right).$$
Let $g: (0,\infty)^{2}\rightarrow (0,\infty)^{2}$ be
$$g((x,y))=\left(\alpha_{2}+\left(\frac{\alpha_{1}}{A_{1}+y}\right) ,x\right).$$
Let $h:(0,\frac{\alpha_{1}}{A_{1}})\times (0,\infty)\rightarrow (0,\infty)^{2}$ be 
$$h((x,y))=\left(y,\frac{\alpha_{1}}{x}-A_{1}\right).$$
Let $h^{-1}:(0,\infty)^{2}\rightarrow (0,\frac{\alpha_{1}}{A_{1}})\times (0,\infty)$ be
 $$h^{-1}((x,y))=\left(\frac{\alpha_{1}}{A_{1}+y},x\right).$$
Now we will show that $h^{-1}\circ g \circ h = f$.
$$h^{-1}\circ g \circ h((x,y))=h^{-1}\circ g\left(\left(y,\frac{\alpha_{1}}{x}-A_{1}\right)\right)=h^{-1}\left(\left(\alpha_{2}+x,y\right)\right)$$
$$= \left(\frac{\alpha_{1}}{A_{1}+y},\alpha_{2}+x\right)=f((x,y)).$$
\par\vspace{.4 cm}
Since the system (11,22) is conjugate to a second order rational difference equation which decouples into two Riccati difference equations with Riccati number different from zero, the unique nonnegative equilibrium is globally asymptotically stable for the system (11,22).
\end{proof}

\section{The System (11,24)}
\begin{thm}
Consider the following system of rational difference equations
$$x_{n+1}=\frac{\alpha_{1}}{A_{1}+y_{n}},\quad n=0,1,\dots,$$
$$y_{n+1}=\frac{\alpha_{2}+x_{n}}{x_{n}},\quad n=0,1,\dots,$$
with $\alpha_1,A_{1},\alpha_{2}>0$ and arbitrary nonnegative initial conditions so that the denominator is never zero.
For this system of rational difference equations there are 2 regions in parametric space with distinct global behavior. The behavior is as follows:
\begin{enumerate}
\item If $\alpha_{1}> \alpha_{2}$, then the unique nonnegative equilibrium $$\left(\frac{\alpha_{1}-\alpha_{2}}{A_{1}+1},\frac{\alpha_{1}+\alpha_{2}A_{1}}{\alpha_{1}-\alpha_{2}}\right)$$ is globally asymptotically stable.
\item If $\alpha_{1}\leq \alpha_{2}$, then for an arbitrary choice of initial conditions $$\lim_{n\rightarrow\infty}(x_{n},y_{n})=\left(0,\infty\right).$$
\end{enumerate}
\end{thm}
\begin{proof}
To find the equilibria we solve:
$$\bar{x}=\frac{\alpha_{1}}{A_{1}+\bar{y}},$$
$$\bar{y}=\frac{\alpha_{2}+\bar{x}}{\bar{x}}.$$
So,
$$\bar{y}\bar{x}=\alpha_{2}+\bar{x},$$
and
$$A_{1}\bar{x}+\bar{y}\bar{x}=\alpha_{1}.$$
Thus,
$$A_{1}\bar{x}+\alpha_{2}+\bar{x}=\alpha_{1}.$$
So, when $\alpha_{1}>\alpha_{2}$ we have the unique nonnegative equilibrium,
$$\bar{x}=\frac{\alpha_{1}-\alpha_{2}}{A_{1}+1},$$
$$\bar{y}=\frac{\alpha_{1}+\alpha_{2}A_{1}}{\alpha_{1}-\alpha_{2}}.$$
When $\alpha_{1}\leq \alpha_{2}$ there are no nonnegative equilibria.
\par\vspace{.4 cm}
Performing linearized stability analysis about the equilibrium $(\bar{x},\bar{y})$ when $\alpha_{1}> \alpha_{2}$ we solve:
$$det\left(\begin{array}{cc}
-\lambda & \frac{-\bar{x}}{A_{1}+\bar{y}}\\               
\frac{1-\bar{y}}{\bar{x}} & -\lambda\\ 
\end{array}
\right)= 0.$$
$$\lambda^{2} + \frac{1-\bar{y}}{A_{1}+\bar{y}}=0.$$
Since $-\bar{y}<1-\bar{y}<0$, we get that $\left|\frac{1-\bar{y}}{A_{1}+\bar{y}}\right|<1$. Thus, when $\alpha_{1}> \alpha_{2}$, the unique nonnegative equilibrium is locally asymptotically stable.
\par\vspace{.4 cm}
Now we will show that our rational system on $(0,\frac{\alpha_{1}}{A_{1}})\times (0,\infty)$ is topologically conjugate to the following second order rational difference equation on $(0,\infty)^{2}$,
$$u_{n+1}=\frac{\alpha_{2}+ \left(\frac{\alpha_{1}}{A_{1}+u_{n-1}}\right) }{ \left(\frac{\alpha_{1}}{A_{1}+u_{n-1}}\right) },\quad n=0,1,\dots.$$
Let $f:(0,\frac{\alpha_{1}}{A_{1}})\times (0,\infty)\rightarrow (0,\frac{\alpha_{1}}{A_{1}})\times (0,\infty)$ be 
$$f((x,y))=\left(\frac{\alpha_{1}}{A_{1}+y},\frac{\alpha_{2}+x}{x}\right).$$
Let $g: (0,\infty)^{2}\rightarrow (0,\infty)^{2}$ be
$$g((x,y))=\left(\frac{\alpha_{2}+ \left(\frac{\alpha_{1}}{A_{1}+y}\right) }{\left(\frac{\alpha_{1}}{A_{1}+y}\right)},x\right).$$
Let $h:(0,\frac{\alpha_{1}}{A_{1}})\times (0,\infty)\rightarrow (0,\infty)^{2}$ be 
$$h((x,y))=\left(y,\frac{\alpha_{1}}{x}-A_{1}\right).$$
Let $h^{-1}:(0,\infty)^{2}\rightarrow (0,\frac{\alpha_{1}}{A_{1}})\times (0,\infty)$ be
 $$h^{-1}((x,y))=\left(\frac{\alpha_{1}}{A_{1}+y},x\right).$$
Now we will show that $h^{-1}\circ g \circ h = f$.
$$h^{-1}\circ g \circ h((x,y))=h^{-1}\circ g\left(\left(y,\frac{\alpha_{1}}{x}-A_{1}\right)\right)=h^{-1}\left(\left(\frac{\alpha_{2}+x}{x},y\right)\right)$$
$$= \left(\frac{\alpha_{1}}{A_{1}+y},\frac{\alpha_{2}+x}{x}\right)=f((x,y)).$$
Since the system (11,24) is conjugate to the second order linear rational difference equation,
 $$u_{n+1}=\frac{A_{1}\alpha_{2}+\alpha_{2}u_{n-1}+ \alpha_{1} }{\alpha_{1}},\quad n=0,1,\dots,$$
we have that whenever $\alpha_{1}> \alpha_{2}$, the unique nonnegative equilibrium is globally asymptotically stable for the system (11,24).
Moreover, whenever $\alpha_{1}\leq \alpha_{2}$ for our system (11,24), then $\lim_{n\rightarrow\infty}y_{n}=\infty$ and $\lim_{n\rightarrow\infty}x_{n}=0$.
\end{proof}

\section{The System (11,28)}
\begin{thm}
Consider the following system of rational difference equations
$$x_{n+1}=\frac{\alpha_{1}}{A_{1}+y_{n}},\quad n=0,1,\dots,$$
$$y_{n+1}=\frac{\alpha_{2}+\gamma_{2}y_{n}}{A_{2}+y_{n}},\quad n=0,1,\dots,$$
with $\alpha_1,A_{1},\alpha_{2},\gamma_{2},A_{2}>0$ and arbitrary nonnegative initial conditions so that the denominator is never zero.
In this case, the unique equilibrium $$\left(\frac{2\alpha_{1}}{2A_{1}+\gamma_{2}-A_{2}+\sqrt{\left(\gamma_{2}-A_{2}\right)^{2}+4\alpha_{2}}},\frac{\gamma_{2}-A_{2}+\sqrt{\left(\gamma_{2}-A_{2}\right)^{2}+4\alpha_{2}}}{2}\right)$$ is globally asymptotically stable.
\end{thm}
\begin{proof}
To find the equilibria we solve:
$$\bar{x}=\frac{\alpha_{1}}{A_{1}+\bar{y}},$$
$$\bar{y}=\frac{\alpha_{2}+\gamma_{2}\bar{y}}{A_{2}+\bar{y}}.$$
So,
$$\bar{y}^{2}+A_{2}\bar{y}=\alpha_{2}+\gamma_{2}\bar{y}.$$
Thus,
$$\bar{y}=\frac{\gamma_{2}-A_{2}+\sqrt{\left(\gamma_{2}-A_{2}\right)^{2}+4\alpha_{2}}}{2},$$
$$\bar{x}=\frac{2\alpha_{1}}{2A_{1}+\gamma_{2}-A_{2}+\sqrt{\left(\gamma_{2}-A_{2}\right)^{2}+4\alpha_{2}}}.$$
Thus, we have a unique nonnegative equilibrium.
\par\vspace{.4 cm}
Performing linearized stability analysis about the equilibrium $(\bar{x},\bar{y})$, we solve:
$$det\left(\begin{array}{cc}
-\lambda & \frac{-\bar{x}}{A_{1}+\bar{y}}\\               
0 & \frac{\gamma_{2}-\bar{y}}{A_{2}+\bar{y}}-\lambda\\ 
\end{array}
\right)= 0.$$
$$\lambda^{2}-\left(\frac{\gamma_{2}-\bar{y}}{A_{2}+\bar{y}}\right)\lambda=0.$$
Now,
$$\frac{\gamma_{2}-\bar{y}}{A_{2}+\bar{y}}=\frac{{\gamma_{2}+A_{2}-\sqrt{\left(\gamma_{2}-A_{2}\right)^{2}+4\alpha_{2}}}}{{\gamma_{2}+A_{2}+\sqrt{\left(\gamma_{2}-A_{2}\right)^{2}+4\alpha_{2}}}},$$
and
$$\left|\frac{{\gamma_{2}+A_{2}-\sqrt{\left(\gamma_{2}-A_{2}\right)^{2}+4\alpha_{2}}}}{{\gamma_{2}+A_{2}+\sqrt{\left(\gamma_{2}-A_{2}\right)^{2}+4\alpha_{2}}}}\right|<1.$$
Thus, the unique nonnegative equilibrium for system (11,28) is locally asymptotically stable.
\par\vspace{.4 cm}
Since the recursive equation governing $y_{n}$ is a Riccati equation which does not depend on $x_{n}$ in any way, the unique equilibrium is globally asymptotically stable for the system (11,28).
\end{proof}

\section{The System (11,32)}
\begin{thm}
Consider the following system of rational difference equations
$$x_{n+1}=\frac{\alpha_{1}}{A_{1}+y_{n}},\quad n=0,1,\dots,$$
$$y_{n+1}=\frac{\alpha_{2}+x_{n}}{A_{2}+x_{n}},\quad n=0,1,\dots,$$
with $\alpha_1,A_{1},\alpha_{2},A_{2}>0$ and arbitrary nonnegative initial conditions so that the denominator is never zero.
In this case, the unique nonnegative equilibrium 
$$\bar{x}=\frac{\alpha_{1}-\alpha_{2}-A_{1}A_{2}+\sqrt{(\alpha_{1}-\alpha_{2}-A_{1}A_{2})^{2}+4\alpha_{1}A_{2}+4\alpha_{1}A_{1}A_{2}}}{2+2A_{1}},$$
$$\bar{y}=\frac{2\alpha_{1}+2\alpha_{1}A_{1}}{\alpha_{1}-\alpha_{2}-A_{1}A_{2}+\sqrt{(\alpha_{1}-\alpha_{2}-A_{1}A_{2})^{2}+4\alpha_{1}A_{2}+4\alpha_{1}A_{1}A_{2}}}-A_{1},$$
is globally asymptotically stable.
\end{thm}
\begin{proof}
To find the equilibria we solve:
$$\bar{x}=\frac{\alpha_{1}}{A_{1}+\bar{y}},$$
$$\bar{y}=\frac{\alpha_{2}+\bar{x}}{A_{2}+\bar{x}}.$$
So,
$$\bar{x}=\frac{\alpha_{1}A_{2}+\alpha_{1}\bar{x}}{A_{1}A_{2}+A_{1}\bar{x}+\alpha_{2}+\bar{x}}.$$
Thus,
$$(1+A_{1})\bar{x}^{2}+(A_{1}A_{2}+\alpha_{2})\bar{x}=\alpha_{1}A_{2}+\alpha_{1}\bar{x}.$$
So the system (11,32) has the unique nonnegative equilibrium,
$$\bar{x}=\frac{\alpha_{1}-\alpha_{2}-A_{1}A_{2}+\sqrt{(\alpha_{1}-\alpha_{2}-A_{1}A_{2})^{2}+4\alpha_{1}A_{2}+4\alpha_{1}A_{1}A_{2}}}{2+2A_{1}},$$
$$\bar{y}=\frac{2\alpha_{1}+2\alpha_{1}A_{1}}{\alpha_{1}-\alpha_{2}-A_{1}A_{2}+\sqrt{(\alpha_{1}-\alpha_{2}-A_{1}A_{2})^{2}+4\alpha_{1}A_{2}+4\alpha_{1}A_{1}A_{2}}}-A_{1}.$$
\par\vspace{.4 cm}
Now we will show that our rational system on $(0,\frac{\alpha_{1}}{A_{1}})\times (0,\infty)$ is topologically conjugate to the following second order rational difference equation on $(0,\infty)^{2}$,
$$u_{n+1}=\frac{\alpha_{2}+ \left(\frac{\alpha_{1}}{A_{1}+u_{n-1}}\right) }{A_{2}+ \left(\frac{\alpha_{1}}{A_{1}+u_{n-1}}\right) },\quad n=0,1,\dots.$$
Let $f:(0,\frac{\alpha_{1}}{A_{1}})\times (0,\infty)\rightarrow (0,\frac{\alpha_{1}}{A_{1}})\times (0,\infty)$ be 
$$f((x,y))=\left(\frac{\alpha_{1}}{A_{1}+y},\frac{\alpha_{2}+x}{A_{2}+x}\right).$$
Let $g: (0,\infty)^{2}\rightarrow (0,\infty)^{2}$ be
$$g((x,y))=\left(\frac{\alpha_{2}+\left(\frac{\alpha_{1}}{A_{1}+y}\right) }{A_{2}+\left(\frac{\alpha_{1}}{A_{1}+y}\right) },x\right).$$
Let $h:(0,\frac{\alpha_{1}}{A_{1}})\times (0,\infty)\rightarrow (0,\infty)^{2}$ be 
$$h((x,y))=\left(y,\frac{\alpha_{1}}{x}-A_{1}\right).$$
Let $h^{-1}:(0,\infty)^{2}\rightarrow (0,\frac{\alpha_{1}}{A_{1}})\times (0,\infty)$ be
 $$h^{-1}((x,y))=\left(\frac{\alpha_{1}}{A_{1}+y},x\right).$$
Now we will show that $h^{-1}\circ g \circ h = f$.
$$h^{-1}\circ g \circ h((x,y))=h^{-1}\circ g\left(\left(y,\frac{\alpha_{1}}{x}-A_{1}\right)\right)=h^{-1}\left(\left(\frac{\alpha_{2}+x}{A_{2}+x},y\right)\right)$$
$$= \left(\frac{\alpha_{1}}{A_{1}+y},\frac{\alpha_{2}+x}{A_{2}+x}\right)=f((x,y)).$$
\par\vspace{.3 cm}
Since the system (11,32) is conjugate to a second order rational difference equation which decouples into two Riccati difference equations with Riccati number different from zero, the unique nonnegative equilibrium is globally asymptotically stable for the system (11,32).
\end{proof}

\section{Conclusion}
This paper is the first in a series of papers which will address, on a case by case basis, the special cases of the following rational system in the plane, labeled system \#11. 
$$x_{n+1}=\frac{\alpha_{1}}{A_{1}+y_{n}},\quad y_{n+1}=\frac{\alpha_{2}+\beta_{2}x_{n}+\gamma_{2}y_{n}}{A_{2}+B_{2}x_{n}+C_{2}y_{n}},\quad n=0,1,2,\dots ,$$
with $\alpha_{1},A_{1}>0$ and $\alpha_{2}, \beta_{2}, \gamma_{2}, A_{2}, B_{2}, C_{2}\geq 0$ and $\alpha_{2}+\beta_{2}+\gamma_{2}>0$ and $A_{2}+B_{2}+C_{2}>0$ and nonnegative initial conditions $x_{0}$ and $y_{0}$ so that the denominator is never zero.
In this article we have focused on the special cases which are reducible to the Riccati difference equation. We have determined the complete picture of the qualitative behavior for special cases of system \#11 which are Riccati or reducible to Riccati, namely the cases numbered $(11,1)$, $(11,2)$, $(11,3)$, $(11,4)$, $(11,5)$, $(11,7)$, $(11,9)$, $(11,10)$, $(11,11)$, $(11,13)$, $(11,17)$, $(11,19)$, $(11,20)$, $(11,22)$, $(11,24)$, $(11,28)$, and $(11,32)$, in the numbering system developed in \cite{cklm}.

\par

\end{document}